\documentclass[12pt]{amsart}
\usepackage{amssymb,verbatim,enumerate,ifthen}
\usepackage{esint}
\usepackage[mathscr]{eucal}
\usepackage[utf8]{inputenc}
\usepackage[T1]{fontenc}
\usepackage{caption}
\usepackage{amsthm,thmtools}
\usepackage{tikz}
\usetikzlibrary{automata,positioning,arrows.meta,shapes}
\makeatletter
\@namedef{subjclassname@2010}{%
  \textup{2020} Mathematics Subject Classification}
\makeatother

\DeclareUnicodeCharacter{1EF3}{\`y}

\newtheorem{thm}{Theorem}

\newtheorem{prop}{Proposition}
\newtheorem{lem}{Lemma}

\theoremstyle{definition}

 \newtheorem{rem}{Remark}
\declaretheorem[name=Example,qed={\lower-0.3ex\hbox{\qedsymbol}}]{exa}

\usepackage{cite}
\newcommand\R{\mathbb{R}}
\newcommand\A{\mathscr{A}}
\renewcommand\P{\mathscr{P}}
\newcommand\N{\mathbb{N}}

\newcommand\MM{\mathbf{M}}
\newcommand\KK{\mathbf{K}}

\newcommand{\abs}[1]{\left| #1 \right| }
\newcommand{\norm}[1]{\left\| #1 \right\| }
\newcommand\dd{\mathbf{d}}

\author[P. Pasteczka]{Pawe\l{} Pasteczka}
\address{Institute of Mathematics \\ Pedagogical University of Krakow \\ Podchor\k{a}\.zych str. 2, 30-084 Krak\'ow, Poland}
\email{pawel.pasteczka@up.krakow.pl}


\newcommand{\operator}[1]{\mathop{\vphantom{\sum}\mathchoice
{\vcenter{\hbox{\LARGE $#1$}}}
{\vcenter{\hbox{\Large $#1$}}}{#1}{#1}}\displaylimits}

\def\Mst_#1^#2{\operator{\mathscr{M}_{\mbox{\scriptsize$\#$}}\!\!}_{#1}^{#2}\,\,}

\DeclareMathOperator{\Erg}{Erg}
\numberwithin{equation}{section}
\def\eq#1{{\rm(\ref{#1})}}
\def\Eq#1#2{\ifthenelse{\equal{#1}{*}}
  {\begin{equation*}\begin{aligned}[]#2\end{aligned}\end{equation*}}
  {\begin{equation}\begin{aligned}[]\label{#1}#2\end{aligned}\end{equation}}}

\title{Invariance property for extended means}
\begin{document}

\begin{abstract}
We study the properties of the mean-type mappings ${\bf M}\colon I^p \to I^p$ of the form
$${\bf M}(x_1,\dots,x_p):=\big(M_1(x_{\alpha_{1,1}},\dots,x_{\alpha_{1,d_1}}),\dots,M_p(x_{\alpha_{p,1}},\dots,x_{\alpha_{p,d_p}})\big),$$
where $p$ and $d_i$-s are positive integers, each $M_i$ is a $d_i$-variable mean on an interval $I \subset \mathbb{R}$, and $\alpha_{i,j}$-s are elements from $\{1,\dots,p\}$. 

We show that, under some natural assumption on $M_i$-s, the problem of existing the unique $\bf M$-invariant mean can be reduced to the ergodicity of the directed graph with vertexes $\{1,\dots,p\}$ and edges $\{(\alpha_{i,j},i) \colon i,j \text{ admissible}\}$. 
\end{abstract}

\maketitle
\section{Introduction}
Invariance property is a very important aspect in the theory of means. There are two classical studies, Lagrange \cite{Lag84} and Gauss \cite{Gau18}, which could be considered as a beginning of this field. It has been extensively studied by many authors since then. For example J.~M.~Borwein and P.~B.~Borwein \cite{BorBor87} extended some earlier ideas \cite{FosPhi84a,Leh71,Sch82} and 
generalized the original iteration to a vector of continuous, strict means of an arbitrary length.
For several recent results about Gaussian product of means, see the papers by Baj\'ak--P\'ales 
\cite{BajPal09b,BajPal09a,BajPal10,BajPal13}, by Dar\'oczy--P\'ales \cite{Dar05a,DarPal02c,DarPal03a}, 
by G{\l}azowska \cite{Gla11b,Gla11a}, by Jarczyk--Jarczyk \cite{JarJar18}, by Matkowski \cite{Mat99b,Mat02b,Mat05,Mat09e}, by Matkowski--P\'ales \cite{MatPal15}, and by the author \cite{Pas16a}.
In the vast majority of these studies, there are assumptions which provide that the invariant mean is uniquely determined. 
There are also few results where this is not the case; see  Der\k{e}gowska--Pasteczka~\cite{DerPas2005},
Matkowski--Pasteczka \cite{MatPas20a,MatPas21}, and Pasteczka \cite{Pas19a,Pas2106}. The main result of these papers is that the uniqueness of invariant means is deeply related to strictness and continuity.

\subsection*{Basic definition and notions} Before we proceed further recall that, for a given $p\in \N$ and an interval $I \subset \R$, a \emph{$p$-variable mean on $I$} is an arbitrary function $M \colon I^p \to I$ satisfying the inequality
\Eq{E:MP}{
\min(x)\le M(x)\le\max(x)\text{ for all }x \in I^p.
} 
Property \eq{E:MP} is referred as a \emph{mean property}. If the inequalities in \eq{E:MP} are strict for every nonconstant vector $x$, then we say that a mean $M$ is \emph{strict}. Moreover, for such objects, we define natural properties like continuity, symmetry (when the value of mean does not depend on the order of its arguments), monotonicity (which states that $M$ is nondecreasing in each of its variables), etc. 

A mean-type mapping is a selfmapping of $I^p$ which has a $p$-variable mean on each of its coordinates. More precisely, $\MM \colon I^p \to I^p$ is called a \emph{mean-type mapping} if $\MM=(M_1,\dots,M_p)$ for some $p$-variable means $M_1,\dots,M_p$ on $I$. In this framework a function $K \colon I^p\to \R$ is called \emph{$\MM$-invariant} if it solves the functional equation $K \circ \MM=K$. Usually we restrict solutions of this equation to the family of means and say about \emph{$\MM$-invariant means}.

\subsection*{Posing the problem} 
There is a natural problem to give a condition to $\MM$ that guarantees the uniqueness of the $\MM$-invariant mean. It turned out that it is equivalent to certain convergence of the sequence of iterations of the self-mapping $\MM$ (which hereafter will be denoted by $\MM^n$), cf. \cite{MatPas21}.   
There are three natural conditions which are proposed in the literature. Namely, if the continuous mean-type mapping $\MM=(M_1,\dots,M_p) \colon I^p\to I^p$ satisfies one of the following three conditions: \\
-- each $M_i$ is a strict mean; \\
-- $\MM$ is \emph{contractive}, that is, $\max \MM(x)-\min \MM(x) < \max(x)-\min(x)$ for every nonconstant vector $x \in I^p$; \\
-- $\MM$ is \emph{weakly contractive}, which states that for every nonconstant vector $x \in I^p$ there exists a natural number $n(x)$ such that 
\Eq{E:weC}{
\max \MM^{n(x)}(x)-\min \MM^{n(x)}(x) < \max(x)-\min(x),}
then there exists exactly one $\MM$-invariant mean (cf. \cite{BorBor87}, \cite{Mat09e}, and \cite{MatPas21}, respectively). Obviously the last condition is the most general, however it is also the most difficult to verify. 

We try to check the last condition in the example.
\begin{exa} Take a quadriple of four-variable power means defined on a set $\R_+:=(0,+\infty)$. Namely, let $\MM \colon \R_+^4 \to \R_+^4$ be a mean-type mapping  given by
\Eq{E:ex1}{
\MM(x,y,z,t):=\bigg(\frac{2xy}{x+y},\sqrt{yz},\frac{z+t}{2}, \sqrt{\frac{t^2+x^2}2}\bigg).
}
Then each coordinate of $\MM$ is a bivariate mean. Furthermore, $\MM$ is not contractive, since this condition voids for all vectors of the form $(a,a,b,b)$. On the other hand, one can prove that each coordinate of
\Eq{*}{
\MM^2(x,y,z,t)&=\Bigg(
\frac{4 \, \sqrt{y z} x y}{ {2 x y + (x+y)\sqrt{y z}}},
 \sqrt[4]{\tfrac14 y z (t + z)^2},\\
&\qquad\qquad \frac{t+z+\sqrt{2 t^{2} + 2 x^{2}}}{4},
\frac{1}{2} \, \sqrt{t^{2} + x^{2} + \frac{8 \, x^{2} y^{2}}{{\left(x + y\right)}^{2}}}\:\Bigg)
}
is a trivariate, strict mean on $I$. Thus, $\MM^2$ is a contractive mean-type mapping. Consequently \eq{E:weC} holds with $n(x):=2$, $\MM$ is weakly contractive, and there exists the unique \mbox{$\MM$-invariant} mean. 
\end{exa}

Observe that $\MM$ has a quite interesting structure. Namely, in each coordinate we take one of the classical means (harmonic, geometric, arithmetic, and quadratic), but we omit some arguments. The aim of this paper is to deliver a robust framework and to prove some natural properties for these sort of mean-type mappings.

\subsection*{Extended means} Now we introduce the essential definition from the point of view of this manuscript. Namely, a $p$-variable mean $M \colon I^p \to I$ is called an \emph{extended mean} if it satisfies a mean property \eq{E:MP} and it is independent on some variable. More precisely, there exists $k \in \{1,\dots,p\}$ such that for all $x,x' \in I^p$ satisfying the equality $x_i=x'_i$ for all $i \in\{1,\dots,p\}\setminus \{k\}$ we have $M(x)=M(x')$. 

For a given $d,p \in \N$, a sequence $\alpha:=(\alpha_1,\dots,\alpha_d) \in \{1,\dots,p\}^d$, and a $d$-variable mean $M \colon I^d\to I$ we define the mean $M^{(p;\alpha)}\colon I^p\to I$
by \Eq{E:parmean}{
M^{(p;\alpha)}(x_1,\dots,x_p):=M(x_{\alpha_1},\dots,x_{\alpha_d}) \text{ for all }(x_1,\dots,x_p)\in I^p.
}
In the case $d<p$, mean $M^{(p;\alpha)}$ is a $p$-variable extended mean on $I$.

For example if $\A\colon \R^2\to \R$ is a bivariate arithemetic mean, $p \ge 3$ and $\alpha=(2,3)$ then $\A^{(p;\alpha)}\colon I^p \to I$ is given by 
\Eq{*}{
\A^{(p;\alpha)}(x_1,\dots,x_p)=\A^{(p;2,3)}(x_1,\dots,x_p)=\tfrac{x_2+x_3}2\text{ for all }(x_1,\dots,x_p)\in I^p.
}

The opposite statement is also valid in some sense. Indeed, for every extended mean $M \colon I^p \to I$, there exist $d<p$, a sequence $\alpha \in \{1,\dots,p\}^d$, and a $d$-variable mean  $M_* \colon I^d\to I$ such that $M=M_*^{(p;\alpha)}$. We are going to study the invariance of mean-type mappings which contain extended means.

\section{Uniformly weak contractive mappings}
It turns out that in this setup it is natural to define a property which is between the contractivity and the weak contractivity. We say that a mean-type mapping 
 $\MM\colon I^p \to I^p$ is \emph{uniformly weak contractive} if there exists a natural number $n_0 \in \N$ (which does not depend on $x$) such that 
\Eq{*}{
\max \MM^{n_0}(x)-\min \MM^{n_0}(x) < \max(x)-\min(x),}
for every nonconstant vector $x \in I^p$. Obviously, every contractive mean-type mapping is uniformly weak contractive, and every uniformly weak contractive mean-type mapping is weakly contractive. Moreover, due to Matkowski-Pasteczka \cite{MatPas20a,MatPas21}, it is known that:\\
(1) for $p=2$ every  weakly contractive  mean-type mapping is uniformly weak contractive with $n_0=2$;\\
(2) for every $p>2$ there exist a weakly contractive  mean-type mapping which is not uniformly weak contractive 

\subsection{Invariance property}
In this section we show a counterpart of the result contained in \cite[Theorem~1]{Mat09e}. The main difference is that we generalize the original setting to the family of uniformly weakly contractive mean-type mappings.

\begin{thm}\label{thm:1}
Let $I \subset \R$ be an interval, $p \in \N$, and $\MM \colon I^p \to I^p$ be the uniformly weak contractive, continuous mean-type mapping. Then 
\begin{enumerate}[{\rm (i)}]
    \item \label{1.A} for every $n\in \N$, the mapping $\MM^n$ is a mean-type mapping; 
    \item \label{1.B} there is a continuous mean $K\colon I^p \to I$ such that the sequence of iterates $(\MM^n)_{n=0}^\infty$ converges, uniformly on compact subsets of $I^p$, to the mean-type mapping $\KK \colon I^p \to I^p$, $\KK=(K_1,\dots,K_p)$ such that 
    \Eq{*}{K_1=\dots=K_p=K;}
    \item \label{1.C} $\KK \colon I^p \to I^p$ is $\MM$-invariant, that is, $\KK =\KK\circ \MM$ or, equivalently, the mean $K$ is $\MM$-invariant;
    \item \label{1.D} $\MM$-invariant mean (mean-type mapping) is unique;
    \item \label{1.E} if $\MM=(M_1,\dots,M_p)$ and all $M_i$-s are strict means, then so is $K$;
    \item \label{1.F} if $\MM=(M_1,\dots,M_p)$ and all $M_i$-s are nondecreasing with respect to each variable then so is $K$;
    \item \label{1.G} if $I=(0,+\infty)$ and $\MM$ is positively homogeneous, then every iterate of $\MM$ and $K$ are positively homogeneous.
    \end{enumerate}
\end{thm}
\begin{proof}
The case when $\MM$ is contractive is due to \cite[Theorem~1]{Mat09e}. Moreover parts \eq{1.A}, and \eq{1.D} are due to \cite[Theorem~2]{MatPas21}, where they were proved for all continuous, weakly contractive mean-type mappings.

Now assume that $\MM\colon I^p \to I^p$ is uniformly weak contractive, and take $k \in \N$ such that $\MM^k$ is contractive.

To show (ii) let $\Lambda$ be a family of all compact subintrevals of $I$. Observe that for every compact subset $X$ of $I^p$, there exists $J \in \Lambda$ such that $X \subset J^p$. Moreover 
\Eq{*}{
I^p=\bigcup_{J \in \Lambda} J^p,}
and $\MM(J^p)\subset J^p$ for every $J \in \Lambda$.
Therefore, it is sufficient to show that the assertion 
\Eq{AAA00}{
\MM^n|_{J^p}\text{ is uniformly convergent to }\KK|_{J^p}\\
\text{and }\KK |_{J^p}=(K,\dots,K)|_{J^p}
}
holds for all $J \in \Lambda$. Consequently, one may assume that $I$ is compact. 

Then, by the contractive part, the assertion \eqref{1.B} holds for the subsequence $(\MM^{kn})_{n=0}^\infty$. More precisely, there exists $\KK \colon I^p \to I^p$ with required properties such that for every $\varepsilon>0$ there exists $n_\varepsilon$ such that \Eq{*}{
\norm{\MM^{kn}(x)-\KK(x)}_\infty<\varepsilon\text{ for all }x \in I^p \text{ and }n\ge n_\varepsilon,
} 
where $\norm{(v_1,\dots,v_p)}_\infty:=\max\{\abs{v_1},\dots,\abs{v_p}\}$.
Equivalently, the property
\Eq{P-ii}{
K(x)-\varepsilon < \min \MM^{s}(x) \le \max \MM^{s}(x) < K(x)+\varepsilon 
}
holds for $s=kn_\varepsilon$. But by the mean-value property \eq{E:MP} we know that 
$M_i(y) \ge \min(y)$ for all $i \in \{1,\dots,p\}$ and $y \in I^p$. Thus 
\Eq{*}{
\min \MM(y)=\min(M_1(y),\dots,M_p(y)) \ge \min(y) \text{ for all }y \in I^p.
}
Thus, by simple induction, the mapping $s \mapsto  \min \MM^{s}(x)$ is nondecreasing. Analogously, we can show that the mapping $s \mapsto  \max \MM^{s}(x)$ is nonincreasing. Whence, by the trivial inequality $\min \MM^{s}(x)\le \max \MM^{s}(x)$, we obtain that \eq{P-ii} holds for all $s\ge kn_\varepsilon$ which completes the proof of~\eq{1.B}. 

Having this proved, for all $x\in I^p$, we obtain
\Eq{*}{
\KK (\MM(x))=\lim_{n \to \infty} \MM^n (\MM(x))=\lim_{n \to \infty} \MM^n (x)=\KK(x),
}
and therefore $\KK\circ\MM=\KK$, which is~\eq{1.C}.


To prove \eq{1.E} note that, under the assumption that all $M_i$-s are strict means, for an arbitrary nonconstant vector $x \in I^p$ we have 
\Eq{*}{
K(x)=K\circ \MM(x) \le \max \big(\MM(x)\big)<\max x.
}
Similarly we can show the inequality $K(x)>\min(x)$.

Now we proceed to the proof of \eq{1.F}. First, let us define the partial ordering $\preceq$ on $I^p$ by
\Eq{*}{
(x_1,\dots,x_p) \preceq (y_1,\dots,y_p)\text{ if and only if }x_i\le y_i\text{ for all }i \in \{1,\dots p\}.
}
Then, since each $M_i$ are nondecreasing, for every $x,y\in I^p$ with $x \preceq y$ we get $\MM(x)\preceq \MM(y)$. Thus, by a simple induction, we also have $\MM^n(x)\preceq \MM^n(y)$ for all $n \in \N$. In the limit case when $n \to \infty$, in view of \eq{1.B}, we obtain that $\KK$ is monotone with respect to $\preceq$, which is \eq{1.F}.

To show the last assertion, take $c>0$ and $x \in (0,+\infty)^p$ arbitrarily. Then, since $\MM$ is positively homogeneous, by \eq{1.B}, we have
\Eq{*}{
\KK(c x)=\lim_{n \to \infty} \MM^n(c x)= \lim_{n \to \infty} c \MM^n(x)=c \KK(x),
}
and thus $K(c x)=c K(x)$.
\end{proof}

\section{$\dd$-averaging mappings}

In this section, we introduce an important subfamily of mean-type mappings and study their properties within this class. Our aim is to reduce the properties of a subclass of mean-type mappings which appeared in the previous section to certain properties of directed graphs.

For the sake of completeness, let us introduce formally $\N:=\{1,\dots\}$, and $\N_p:=\{1,\dots,p\}$ (where $p\in \N$). Then, for $p \in \N$ and a vector $\dd=(d_1,\dots,d_p)\in \N^p$, let $\N_p^\mathbf{d}:=\N_p^{d_1}\times\dots\times \N_p^{d_p}$.
Using this notations, a sequence of means $\MM=(M_1,\dots,M_p)$ is called \emph{$\dd$-averaging} mapping on~$I$ if each $M_i$ is a $d_i$-variable mean on $I$. 

For a $\dd$-averaging  mapping $\MM$ and a vector of indexes $\alpha =(\alpha_1,\dots,\alpha_p)\in \N_p^{d_1}\times\dots\times \N_p^{d_p}=\N_p^\dd$ define a mean-type mapping $\MM_\alpha \colon I^p \to I^p$ by 
\Eq{*}{
\MM_\alpha:=\Big(M_1^{(p;\alpha_1)},\dots,M_p^{(p;\alpha_p)}\Big);
}
recall that $M_i^{(p,\alpha_i)}$-s were defined in \eq{E:parmean}. In the more explicit form we have
\Eq{Explicit}{
[\MM_\alpha](x_1,\dots,x_p)&=\Big(M_i^{(p,\alpha_i)}(x_1\dots,x_p)\Big)_{i=1}^p\\
&=\Big(M_i\big(x_{\alpha_{i,1}},\dots,x_{\alpha_{i,d_i}}\big)\Big)_{i=1}^p\\
&=\Big(M_1\big(x_{\alpha_{1,1}},\dots,x_{\alpha_{1,d_1}}\big),\dots,M_p\big(x_{\alpha_{p,1}},\dots,x_{\alpha_{p,d_p}}\big)\Big).
}

Few examples of this sort of mean-types mapping are presented in the last section. 
We aim to study the family of means which are \mbox{$\MM_\alpha$-invariant}. The important part of our consideration will use some facts from graph theory. 

\begin{rem}
Observe that for each element on $\alpha\in\N_p^\dd$ we can recover the value of $p$ (based on the vector $\dd$)\, but we cannot do the same for the single element of $\alpha_i \in \N_p^{d_i}$ (in this example $i \in \{1,\dots,p\}$). Therefore, it is natural to use notations $\MM_\alpha$ and $M_i^{(p;\alpha_i)}$. 
\end{rem}

\subsection{General properties of directed graphs}
Now we recall some elementary facts concerning graphs. For details, we refer the reader to the classical book \cite{GraKnuPat89}.

A \emph{digraph} is a pair $G=(V,E)$, where $V$ is a finite set of vertexes, and $E\subset V \times V$ is a set of edges. For each $v \in V$ we denote by $N_G^-(v)$ and $N_G^+(v)$ sets of \emph{in-neighbors} and \emph{out-neighbors}, respectively. More precisely $N_G^-(v)=\{w \in V \colon (w,v)\in E\}$ and $N_G^+(v)=\{w \in V \colon (v,w)\in E\}$.  The edges of the form $(v,v)$ for $v \in V$ are called \emph{loops}.

A sequence $(v_0,\dots,v_n)$ of elements in $V$ such that $(v_{i-1},v_{i})\in E$ for all $i \in \{1,\dots,n\}$ is called a \emph{walk} from $v_0$ to $v_n$. The number $n$ is a \emph{length} of the walk. If for all $v, w \in V$ there exists a walk from $v$ to $w$, then $G$ is called \emph{irreducible}.

A \emph{cycle} in a graph is a non-empty walk in which only the first and last vertices are equal. A directed graph is said to be \emph{aperiodic} if there is no integer $k > 1$ that divides the length of every cycle of the graph. A~graph which is simultaneously irreducible and aperiodic is called \emph{ergodic}. We also need two lemmas which will be useful in the remaining part of this paper
\begin{lem} \label{new}
Let $G=(V,E)$ be an ergodic digraph. Then there exists $q_0$ such that for all $q \ge q_0$, and $v,w \in V$ there exists a walk from $v$ to $w$ of length exactly $q$.
\end{lem}
\begin{proof}
First, since $G$ is ergodic, there exist $k>1$ and cycles $(C_i)_{i=1}^k$ in $G$ of length $(m_i)_{i=1}^k$, respectively, such that $\gcd(m_1,\dots,m_k)=1$. Then there exists $n_0\in \N$, such that every number $n \ge n_0$ can be expressed as 
\Eq{*}{
n=m_1\alpha_1+\dots+m_k\alpha_k
}
for some nonnegative integers $\alpha_1,\dots,\alpha_k$ (see for example \cite{Bra42}). Next, for $v,w \in V$, denote 
\Eq{*}{
P_{v,w}:=\bigg\{s \in \N \colon 
\begin{matrix} 
\text{ there exists a walk from }v\text{ to }w\text{ of length }s\\
\text{ which contains every vertex in }G
\end{matrix} \bigg\}
}

Since $G$ is irreducible, for all $v,w \in V$ there exists a walk $A_{vw}$ from $v$ to $w$, which contains all vertices in $G$ (vertices may appear many times) -- we denote its length by $l_{v,w}$. Then $A_{vw}$ has a nonempty intersection with all cycles $(C_i)_{i=1}^k$ (since it contains all vertexes). Consequently, we may extend $A_{vw}$ by "taking a detour" thought any cycle $C_i$, also many times. Thus
\Eq{*}{
\big\{l_{v,w}+m_1\alpha_1+\dots+m_k\alpha_k \ \colon\  \alpha_1,\dots,\alpha_k \ge 0\big\} \subseteq  P_{v,w}.
}
Therefore $q \in P_{v,w}$ for all $q \ge l_{v,w}+n_0$. Consequently $q \in \bigcap\limits_{v,w \in V} P_{v,w}$ for every $q \ge \max\big\{l_{v,w}\colon v,w\in V\big\}+n_0:=q_0$, which completes the proof.
\end{proof}

\begin{lem} \label{lem:1}
For an ergodic digraph $G=(V,E)$ define $T_G \colon \{-1,0,1\}^V \to \{-1,0,1\}^V$ as follows: for an arbitrary $c \colon V \to \{-1,0,1\}$ and $v \in V$ we set
\Eq{E:TG}{
T_G(c)(v):=\begin{cases} 
1 & \text{ if }c(w)=1 \text{ for all }w \in N_G^-(v);\\
-1 & \text{ if }c(w)=-1 \text{ for all }w \in N_G^-(v);\\
0 & \text{ otherwise}.
\end{cases}
}

Then for every function $c_0 \colon V \to \{-1,0,1\}$ there exists a number $\bar c \in \{-1,0,1\}$ such that
$T_G^{n}(c_0) \equiv \bar c$ for all $n \ge 3^{|V|}$. 

Moreover $\bar c=0$ unless $c_0 \equiv 1$ or $c_0 \equiv -1$.
\end{lem}
\begin{proof}
For the sake of brevity, for all $n\in\N$, we denote briefly $c_n:=T_G^n(c_0)$. First observe that if $c_{n_0}$ is constant for some $n_0 \in \N$ then, since constant functions are fixed points of $T_G$, we obtain $c_n = c_{n_0}$ for all $n \ge n_0$. Thus, in order to show the main part of the statement, it is sufficient to show that $c_{n_0}$ is constant for some $n_0 \in \{0,\dots,3^{|V|}\}$. 

By $c_{n+1}=T_G(c_n)$, since the range of $T$ has $3^{|V|}$ elements, we know that there exist $\alpha,n_0 \in \{1,\dots,3^{|V|}\}$ such that 
\Eq{E:314}{
c_{n+\alpha}=c_n\text{ for all }n\ge n_0.
}

Next, we show that $c_{n_0}$ is a constant function. First, since $c_{n_0}(V)$ is a nonempty subset of $\{-1,0,1\}$, we know that at least one of the following conditions is valid: \Eq{*}{
{\rm\sc A.}\ c_{n_0}(V)=\{0\}, \quad\text{\rm\sc B. }1\in c_{n_0}(V),  \quad\text{\rm\sc C. }{-1}\in c_{n_0}(V). 
}
It splits our proof to three (possibly overlapping) cases.
\vskip1mm
\noindent{\sc Case A.} Once we have $c_{n_0}(V)=\{0\}$, that is $c_{n_0}\equiv 0$, then $T_G^n(c_0)=T^{n-n_0}(c_{n_0})\equiv 0$ for all $n \ge n_0$ and we are done.

\vskip1mm
\noindent{\sc Case B.} If $1 \in c_{n_0}(V)$ then take $v_0\in V$ such that $c_{n_0}(v_0)=1$. By \eq{E:314} we have 
\Eq{E:329}{
c_{n_0+k\alpha}(v_0)=1\text{ for all }k \in \N.
}

Next, define a sequence $(V_n)_{n=0}^\infty$ of sets of vertexes (that is subsets of $V$) by $V_0:=\{v_0\}$, and
\Eq{*}{
V_p:=\big\{v \in V \colon \text{there exists a walk from }v\text{ to }v_0\text{ of length }p\big\},\ p\ge 1.
}
Observe that $w \in V_{p+1}$ if and only if there exists an edge from $w$ to some vertex in $V_p$, that is $V_{p+1}=N_G^-(V_p)$ ($p\ge 0$). Thus $c_n(v)=1$ for all $v \in V_p$ implies $c_{n-1}(v)=1$ for all $v \in V_{p+1}$.
By a simple induction, in view of \eq{E:329}, one gets 
\Eq{E:333}{
c_{0}(v)=1\text{ for all }k \in \N\text{ and }v \in V_{n_0+k\alpha}.
}

By Lemma~\ref{new}, since $G$ is ergodic, there exists $q_0$ such that $V_q=V$ for all $q\ge q_0$. Take $k_0$ such that $n_0+k_0\alpha \ge q_0$. Then, by \eq{E:333}, we have $c_0(v)=1$ for all $v\in V$.

\vskip1mm
\noindent{\sc Case C.} Whenever $-1 \in c_{n_0}(V)$ then, analogously to the previous case, one gets $c_0\equiv-1$.

\medskip
Binding all the latter cases, we have proved that $c_{n_0}$ is a constant function, that is, $c_{n_0} \equiv \bar c$ for some $\bar c \in \{-1,0,1\}$. Since $n_0 \le 3^{|V|}$, and $T_G(c_{n_0})=c_{n_0}$, we obtain 
\Eq{*}{
T_G^n(c_0)=T_G^{n-n_0}\big(T_G^{n_0}(c_0)\big)=T_G^{n-n_0}(c_{n_0})=c_{n_0}\equiv \bar c \text{ for all }n\ge 3^{|V|}.
}
Moreover if $\bar c \ne 0$ then we are not in case A, and therefore $c_0 \equiv 1$ or $c_0 \equiv -1$ as it has been proved in cases B and C, respectively.
\end{proof}

After this extensive introduction, for a given $p\in\N$, $\dd=(d_1,\dots,d_p)\in \N^p$, and $\alpha \in \N_p^\dd$, we define the \emph{$\alpha$-incidence} graph $G_\alpha=(V_\alpha,E_\alpha)$ as follows: $V_\alpha:=\N_p$ and $E_\alpha:=\{(\alpha_{i,j},i) \colon i \in \N_p \text{ and }j \in \N_{d_i}\}$. 

In view of \eq{Explicit} we get that $x_k$ appears in $[\MM_\alpha]_i$ as an argument if and only if $G_\alpha$ contains the edge from $k$ to $i$. We are going to prove that the natural assumption to warranty that \mbox{$\MM_\alpha$-invariant} mean is uniquely determined, is that $G_\alpha$ is ergodic. Therefore, for $p \in \N$ and $\dd \in \N^p$, we set 
\Eq{*}{
\Erg(\dd):=\{\alpha \in \N_p^\dd \colon G_\alpha \text{ is ergodic}\}.
}

\subsection{Graphs of averaging mappings}
Now we show the first nontrivial result referring directly to $\dd$-averaging mappings.

\begin{prop}\label{prop:1}
Let $I \subset \R$ be an interval, $p \in \N$, $\dd \in \N^p$, $\alpha \in \Erg(\dd)$, and $\MM=(M_1,\dots,M_p)$ be a $\dd$-averaging mapping. If all $M_i$-s are strict means, then $\MM_\alpha$ is uniformly weak contractive.

Moreover, for all $x \in I^p$, either $\MM_\alpha^{3^p}(x)$ is a constant vector or 
\Eq{prop:1E}{
\min(x)<\min \MM_\alpha^{3^p}(x)<\max \MM_\alpha^{3^p}(x)<\max x.
}
\end{prop}

\begin{proof}
Let $\Gamma$ be a family of all nonconstant vectors in $I^p$. For $n \in \{0,1,\dots\}$ and $x=(x_1,\dots,x_p)\in \Gamma$ define $c_{x,n} \colon \N_p \to \{-1,0,1\}$ by
\Eq{*}{
c_{x,0}(i)&:=\begin{cases}
1 & \text{ if }x_i=\max(x);\\
-1 & \text{ if }x_i=\min(x);\\
0 & \text{ otherwise};
\end{cases}\\
c_{x,n}(i)&:=\begin{cases}
1 & \text{ if }[\MM_\alpha^n(x)]_i=\max(x);\\
-1 & \text{ if }[\MM_\alpha^n(x)]_i=\min(x);\\
0 & \text{ otherwise}
\end{cases} &\text{ for }n\ge 1.
}
 Then $c_{x,n}(i)=1$ for some $x \in \Gamma$ and $n \ge 1$ yields $[\MM_\alpha^n(x)]_i=\max(x)$, thus 
$$M_i\big([\MM_\alpha^{n-1}(x)]_{\alpha_{i,1}},\dots,[\MM_\alpha^{n-1}(x)]_{\alpha_{i,d_i}}\big)=\max(x).$$ Since $M_i$ is a strict mean, one has 
\Eq{*}{
[\MM_\alpha^{n-1}(x)]_{\alpha_{i,j}}=\max(x)\text{ for all }j \in \N_{d_i}.
}
In the other words, $c_{x,n-1}(\alpha_{i,j})=1$ for all $j \in \N_{d_i}$. Therefore, if $c_{x,n}(i)=1$ for some $x \in \Gamma$ and $n \in \N$, then $c_{x,n-1}(v)=1$ whenever $v \in N_{G_\alpha}^-(i)$. 

The converse implication is also valid. Indeed, if $c_{x,n-1}(v)=1$ for all $v \in N_{G_\alpha}^-(i)$ then $[\MM_\alpha^{n-1}(x)]_k=\max(x)$ for all $k \in \{\alpha_{i,1},\dots,\alpha_{i,d_i}\}$ which yields $[\MM_\alpha^{n}(x)]_i=\max(x)$. Similarly $c_{x,n}(i)=-1$ if and only if $c_{x,n-1}(v)=-1$ for all $v \in N_{G_\alpha}^-(i)$.

Consequently, for every $x \in \Gamma$ and $n \ge 0$ we have $c_{x,n+1}=T_{G_\alpha}(c_{x,n})$, where $T_{G_\alpha}$ is defined by \eq{E:TG}. Thus $c_{x,n}=T_{G_\alpha}^n(c_{x,0})$. 
By Lemma~\ref{lem:1}, $c_x:=c_{x,3^p}$ is a constant function for every $x \in \Gamma$. Now we have three cases.

First, if $c_x \equiv -1$ or $c_x \equiv 1$, then $\MM_\alpha^{3^p}(x)$ is a constant vector and, since $x$ is nonconstant, one has
\Eq{E:Ncx}{
\max \MM_\alpha^{3^p}(x)-\min \MM_\alpha^{3^p}(x)<\max x -\min x.
}
Next, if $c_x \equiv 0$ then $\min(x) < [\MM_\alpha^{3^p}(x)]_i < \max(x)$ for all $i \in \N_p$, which also yields \eq{E:Ncx}. Thus, by the mean-value property, one gets
$$\max \MM_\alpha^{n}(x)-\min \MM_\alpha^{n}(x)<\max x -\min x \quad \text{ for all }x \in \Gamma \text{ and }n\ge 3^p,$$
which implies the uniform weak conctractivity of $\MM_\alpha$.
Finally note that, based on the  proof above, we can easily deduce the moreover part.
\end{proof}

\subsection{Invariance problem}
Following the convention used by several authors (for example, it was used in \cite{Mat09e}) and Theorem~\ref{thm:1} above, we are going to bind several results in a single theorem. The idea beyond this result (and simultaneously the sketch of its proof) is to bind Proposition~\ref{prop:1} (which states that, under certain conditions, $\MM_\alpha$ is uniformly weak contractive) and Theorem~\ref{thm:1} (which provides a number of properties of such mappings). Before we formulate this result, which should be considered as the most important outcome of this paper, let us underline two issues. First, we shuffle the order of items in these results (in our opinion, they are more natural). Second, in some parts we need to give some effort in the proof, since we cannot apply Theorem~\ref{thm:1} directly.

\begin{thm}\label{thm:2}
Let $I \subset \R$ be an interval, $p \in \N$, $\dd \in \N^p$, and $\MM=(M_1,\dots,M_p)$ be a $\dd$-averaging mapping on $I$ such that all $M_i$-s are continuous and strict. Then, for all $\alpha \in \Erg(\dd)$,
\begin{enumerate}[{\rm (a)}]
    \item \label{2.A} there exists the unique $\MM_\alpha$-invariant mean $K_\alpha \colon I^p \to I$; 
    \item \label{2.B} $K_\alpha$ is continuous;
    \item \label{2.C} $K_\alpha$ is strict;
    \item \label{2.D} $(\MM_\alpha^n)_{n=1}^\infty$ converges, uniformly on compact subsets of $I^p$, to the mean-type mapping $\KK_\alpha \colon I^p \to I^p$, $\KK_\alpha=(K_\alpha,\dots,K_\alpha)$;
    \item \label{2.E} $\KK_\alpha \colon I^p \to I^p$ is $\MM_\alpha$-invariant, that is $\KK_\alpha =\KK_\alpha\circ \MM_\alpha$;
    \item \label{2.F} if $M_1,\dots,M_p$ are nondecreasing with respect to each variable, then so is $K_\alpha$;
    \item \label{2.G} if $I=(0,+\infty)$ and $M_1,\dots,M_p$ are positively homogeneous, then every iterate of $\MM_\alpha$ and $K_\alpha$ are positively homogeneous.
\end{enumerate}
\end{thm}
\begin{proof}
By Proposition~\ref{prop:1} we know that $\MM_\alpha$ is uniformly weakly contractive. Then the vast majority of the proof is based on Theorem~\ref{thm:1}; more precisely: 
\eqref{1.D} yields \eqref{2.A}, \eqref{1.B} implies \eqref{2.B} and \eqref{2.D}, and \eqref{1.C} implies \eqref{2.E}. 

Moreover, for all $i \in \N_p$, if $M_i$ is nondecreasing then the mapping \Eq{*}{I^p \ni (x_1,\dots,x_p) \mapsto M_i(x_{\alpha_{i,1}},\dots,x_{\alpha_{i,d_i}})} is also nondecreasing (in each of its variable), and therefore $[\MM_\alpha]_i$ is nondecreasing. Since $i \in \N_p$ is arbitrary, in view of Theorem~\ref{thm:1} part \eqref{1.F}, we get \eqref{2.F}.  Analogously, using part \eqref{1.G} of the same result, one can prove \eqref{2.G}.

At this stage \eqref{2.C} is the only remaining part to be proved, whence we need to prove that $K_\alpha$ is a strict mean on $I$. To this end, take a nonconstant vector $x \in I^p$. We show that $K_\alpha(x)<\max(x)$ (the proof of the second inequality is analogous).

Due to the moreover part of Proposition~\ref{prop:1} we have that either  \eq{prop:1E} holds or $\MM_\alpha^{3^p}(x)$ is a constant vector.  In the first case, since $K_\alpha$ is \mbox{$\MM_\alpha$-invariant}, using the inequality $K_\alpha=K_\alpha\circ \MM_\alpha^{3^p}$ and the mean-value property of $K_\alpha$, we obtain 
\Eq{*}{
K_\alpha(x)=K_\alpha \circ \MM_\alpha^{3^p} (x)\le\max \big(\MM_\alpha^{3^p} (x)\big) < \max(x).
}

If $\MM_\alpha^{3^p}(x)$ is a constant vector (that is $\MM_\alpha^{3^p}(x)=\KK_\alpha(x)$) then let $n_0\in\{0,\dots,3^p\}$ be the smallest number such that $\MM_\alpha^{n_0}(x)=\KK_\alpha(x)$.

Obviously $n_0>0$ since $x$ is nonconstant. Moreover, $y:=\MM_\alpha^{n_0-1}(x)$ is a nonconstant vector with $\KK_\alpha(x)=\MM_\alpha(y)$. 

Then $y_k=\min(y)$ for some $k \in\N_p$. Since $G_\alpha$ is irreducible, we have $(k,i) \in E_\alpha$ for some $i \in \N_p$. 
Therefore, by the definition of $G_\alpha$, one gets $\alpha_{i,j}=k$ for some $j \in \N_{d_i}$. 
Then $\max(y_{\alpha_{i,1}},\dots,y_{\alpha_{i,d_i}}) \le \max(y)$ and $y_{\alpha_{i,j}}=y_k=\min(y)<\max(y)$. Therefore, since $M_i$ is strict,
\Eq{*}{
K_\alpha(x)=[\KK_\alpha(x)]_i=[\MM_\alpha(y)]_i=M_i(y_{\alpha_{i,1}},\dots,y_{\alpha_{i,d_i}})<\max(y).
}
However, since $\MM_\alpha$ is a mean-type mapping, we obtain \Eq{*}{
\max(y)=\max(\MM_\alpha^{n_0-1}(x))\le \max(x),} and thus we get $K_\alpha(x)<\max(x)$.

Similarly, one can prove the inequality $K_\alpha(x)>\min(x)$. Since $x$ is an arbitrary nonconstant vector in $I^p$, we obtain that $K_\alpha$ is a strict mean, which was the last unproved part of this statement.
\end{proof}

\section{Applications and examples}
\subsection{An application to functional equations}
As a first application, we solve the functional equation $F \circ \MM_\alpha=F$. Obviously, under standard conditions, it has a unique solution in the family of means. We show that if we extend the considered family to all functions which are continuous on the
diagonal, we are able to follow the pattern of invariant means. 

\begin{thm}
Let $I \subset \R$ be an interval, $p \in \N$, $\dd \in \N^p$, $\alpha \in \Erg(\dd)$, and $\MM=(M_1,\dots,M_p)$ be a $\dd$-averaging mapping on $I$ such that all $M_i$-s are strict. 

A function $F\colon I^p \to \R$ that is continuous on the
diagonal $\Delta(I^p):=\{(u_1,\dots,u_p)\in I^p \colon u_1=\dots=u_p\}$ is invariant with respect to the
mean-type mapping $\MM_\alpha$, i.e. $F$ satisfies the functional equation 
\Eq{*}{
F\circ \MM_\alpha=F
}
if, and only if, there is a continuous function $\varphi \colon I \to \R$ such that $F=\varphi \circ K_\alpha$, where $K_\alpha \colon I^p \to I$ it the unique $\MM_\alpha$-invariant mean.
\end{thm}
\begin{proof}
Take an $\MM_\alpha$-invariant function $F \colon I^p \to \R$ that is continuous on the diagonal.
Then, for all $n \in \N$, we have $F \circ \MM_\alpha^n=F$. In the limit case, by Theorem~\ref{thm:2} part (d), since $F$ is continuous on the diagonal, we get $F=F \circ \KK_\alpha$. Thus $F=\varphi\circ K_\alpha$ for $\varphi(x):=F(x,\dots,x)$.

Conversely, if $F=\varphi \circ K_\alpha$ then, since $K_\alpha$ in $\MM_\alpha$-invariant, for all $x \in I^p$ we get 
\Eq{*}{
F \circ\MM_\alpha(x)=\varphi \circ K_\alpha \circ \MM_\alpha(x)=\varphi \circ K_\alpha(x)=F(x),}
which completes the proof
\end{proof}

\subsection{Classical application of Theorem~\ref{thm:2}} 
In this section we apply Theorem~\ref{thm:2} to show that the mean-type mapping given by \eq{E:ex1} has the unique invariant mean. This was one of the motivations to write this paper. In this and the subsequent section, all mean-type mappings contain only power means (since the means that build the mean-type mapping $\MM$ do not affect the convergence of the sequence of iterates $(\MM_\alpha^n)$ provided that they are all strict). Recall that the $n$-variable power mean of order $s$ is defined by
\Eq{*}{
\P_s(x_1,\dots,x_n)=\begin{cases}
\Big(\dfrac{x_1^s+\cdots+x_n^s}{n}\Big)^{1/s} &\quad \text{ if } s\in\R\setminus\{0\}, \\[2mm]                                              
\sqrt[n]{x_1\cdots x_n} &\quad \text{ if } s = 0,
\end{cases}
}
where $n \in \N$ and $x_1,\dots,x_n \in \R_+$. For simplicity, let us assume that all means in this section are on $\R_+$.

\begin{exa}\label{exa:ergodic}
Let $\MM \colon \R_+^4 \to \R_+^4$ be given by
\eq{E:ex1}. We show that there exists a unique $\MM$-invariant mean $K \colon \R_+^4 \to \R_+$. Additionally, $K$ is continuous and strict.
\begin{figure}[h] 
\centering
\captionsetup{justification=centering}
\begin{tikzpicture}[>={Stealth[width=6pt,length=9pt]},
node distance=25mm,auto]
\node[state,inner sep=3pt,minimum size=1pt] (b) at (1,1) {$2$};
\node[state,inner sep=3pt,minimum size=1pt] (c) at (0,0) {$3$};
\node[state,inner sep=3pt,minimum size=1pt] (d) at (1,-1) {$4$};
\node[state,inner sep=3pt,minimum size=1pt] (a) at (2,0) {$1$};
\path[->] (a) edge [loop right] node {} (a);
\path[->] (b) edge [loop above] node {} (b);
\path[->] (d) edge [loop below] node {} (d);
\path[->] (c) edge [loop left] node {} (c);
\path[->] (a) edge [bend left=35] node {} (d);
\path[->] (b) edge [bend left=35] node {} (a);
\path[->] (c) edge [bend left=35] node {} (b);
\path[->] (d) edge [bend left=35] node {} (c);
\end{tikzpicture}
\caption{Graph $G_\alpha$ related to Example~\ref{exa:ergodic}.}
\end{figure}
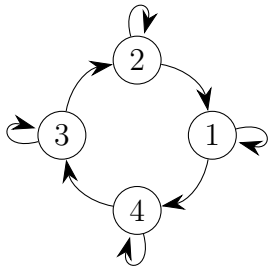

Indeed, in the framework of $\dd$-averaging mappings, we express $\MM$ defined in \eq{E:ex1} as $\bar\MM_\alpha$, where $\bar\MM$ consists of bivariate power means, that is 
\Eq{*}{
\bar\MM=(\P_{-1},\P_0,\P_1,\P_2),\text{ and }\alpha=\big((1,2),(2,3),(3,4),(4,1)\big).
}
The vector $\dd$ contains the lengths of the elements in $\alpha$ (since $\alpha \in \N_4^\dd$), thus $\dd=(2,2,2,2)$.
Obviously all means in $\bar\MM$, being power means, are continuous an strict. 
Moreover, the $\alpha$-incidence graph is aperiodic (since every vertex has a loop) and irreducible (since $(4321)$ is its Hamiltonian cycle). Consequently the $\alpha$-incidence graph is ergodic. 

Thus, in view of Theorem~\ref{thm:2}, there exists exactly one $\MM$-invariant mean $K \colon \R_+^4 \to \R_+$. Moreover, by the same theorem, we know that it is continuous and strict.
\end{exa}

\subsection{Mean-type mappings without ergodic incidence graph}

In the last section, we show a few difficulties that arise in this setting. 
Moreover, in each example we present the incidence graph, which would help us to understand the problems appearing when it comes to deal with the invariance problem.

In the first example we show what happens if the incidence graph is disconnected.

\begin{exa}[Disconnected incidence graph]\label{exa:dig}
Let $p=4$, 
\Eq{*}{
\dd&=(2,2,2,2), \\
\alpha&=\big((1,2),(1,2),(3,4),(3,4)\big) \in \N_4^\dd,\\
\MM&=(\P_{-1},\P_1,\P_{-1},\P_1).
}
Then the mean-type mapping $\MM_\alpha \colon \R_+^4\to \R_+^4$ is of the form
\Eq{*}{
\MM_\alpha(x,y,z,t)&=\bigg(\frac{2xy}{x+y},\frac{x+y}2,\frac{2zt}{z+t},\frac{z+t}2\:\bigg).
}

\begin{figure}[h]
\centering
\captionsetup{justification=centering}
\begin{tikzpicture}[>={Stealth[width=6pt,length=9pt]},
node distance=15mm,auto]
\node[state,inner sep=3pt,minimum size=1pt] (b) at (0,0) {$2$};
\node[state,inner sep=3pt,minimum size=1pt] (c) at (2,0) {$3$};
\node[state,inner sep=3pt,minimum size=1pt] (d) at (2,2) {$4$};
\node[state,inner sep=3pt,minimum size=1pt] (a) at (0,2) {$1$};
\path[->] (a) edge [bend left] node {} (b);
\path[->] (a) edge [loop left] node {} (a);
\path[->] (b) edge [bend left] node {} (a);
\path[->] (b) edge [loop left] node {} (b);
\path[->] (d) edge [bend left] node {} (c);
\path[->] (c) edge [bend left] node {} (d);
\path[->] (c) edge [loop right] node {} (c);
\path[->] (d) edge [loop right] node {} (d);
\end{tikzpicture}
\caption{Graph $G_\alpha$ related to Example~\ref{exa:dig}.}
\end{figure}
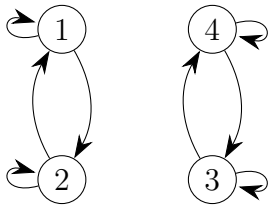

Observe that, in this case, $\MM_\alpha$ is not weakly contractive, for example 
\Eq{*}{
\MM_\alpha^n(1,1,2,2)=(1,1,2,2)\text{ for all }n\in\N.
}
As a matter of fact, we can split the mapping $\MM_\alpha$ into two bivariate mappings. Then, using the classical result, which says that the arithmetic-harmonic mean coincides with the geometric mean, we obtain
\Eq{*}{
\lim_{n\to\infty}\MM^n_\alpha(x,y,z,t)&=(\sqrt{xy},\sqrt{xy},\sqrt{zt},\sqrt{zt}).
}
As a result we have two natural $\MM_\alpha$-invariant means. Namely
\Eq{*}{
K_1(x,y,z,t)=\sqrt{xy}\text{ and }K_2(x,y,z,t)=\sqrt{zt}.
}
Therefore $K(x,y,z,t)=S(\sqrt{xy},\sqrt{zt})$ is $\MM_\alpha$-invariant for every mean $S \colon \R_+^2 \to \R_+$. Note that we cannot exclude that there are other \mbox{$\MM_\alpha$-invariant} means, however all continuous solutions are of this form.

\end{exa}
In next two examples we deal with the weakly connected graphs which are not irreducible. We do believe that it is possible to generalize these examples to a result which covers weakly connected graphs. However, at this stage, it is a conjecture.

\begin{exa}[Weakly connected incidence graph I]\label{exa:Wcig1}
Let $p=4$,
\Eq{*}{
\dd&=(3,3,3,3),\\
\alpha&=\big((1,2,3),(1,2,3),(1,2,3),(1,2,3)\big)\in \N_4^\dd, \\
\MM&=(\P_{-1},\P_0,\P_1,\P_2).
}
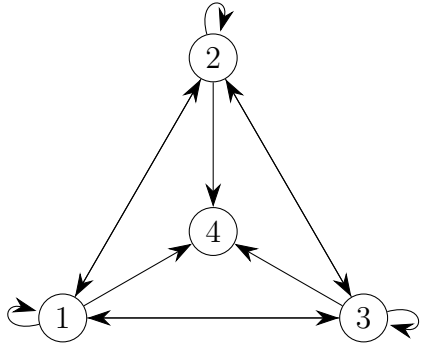
\begin{figure}[h]
\centering
\captionsetup{justification=centering}
\begin{tikzpicture}[>={Stealth[width=6pt,length=9pt]},
node distance=15mm,auto]
\node[state,inner sep=3pt,minimum size=1pt] (b) at (2,3.46) {$2$};
\node[state,inner sep=3pt,minimum size=1pt] (c) at (4,0) {$3$};
\node[state,inner sep=3pt,minimum size=1pt] (d) at (2,1.15) {$4$};
\node[state,inner sep=3pt,minimum size=1pt] (a) at (0,0) {$1$};
\path[->] (a) edge [loop left] node {} (a);
\path[->] (b) edge [loop above] node {} (b);
\path[->] (c) edge [loop right] node {} (c);
\path[->] (a) edge  node {} (b);
\path[->] (c) edge node {} (a);
\path[->] (c) edge  node {} (b);
\path[->] (a) edge node {} (c);
\path[->] (b) edge  node {} (a);
\path[->] (b) edge  node {} (c);
\path[->] (a) edge  node {} (d);
\path[->] (b) edge node {} (d);
\path[->] (c) edge node {} (d);
\end{tikzpicture}
\caption{Graph $G_\alpha$ related to Example~\ref{exa:Wcig1}.}
\end{figure}

Then the mean-type mapping $\MM_\alpha \colon \R_+^4\to \R_+^4$ is of the form
\Eq{*}{
\MM_\alpha(x,y,z,t)=\bigg(\frac{3xyz}{xy+yz+zx},\sqrt[3]{xyz}, 
\frac{x+y+z}3,\sqrt{\frac{x^2+y^2+z^2}3}\:\bigg).
}
Thus $\MM_\alpha$ does not depend on the last coordinate. As a result, we cannot claim that the $\MM_\alpha$-invariant mean is, for example, monotone or strict. In this case, however, the $\MM_\alpha$-invariant mean is uniquely determined, since $\MM_\alpha$ restricted to the first three variables (in both domain and values) admit the unique invariant mean.
\end{exa}

\begin{exa}[Weakly connected incidence graph II]\label{exa:Wcig2}
Let $p=4$,
\Eq{*}{
\dd&=(2,2,2,2),\\
\alpha&=\big((1,2),(1,2),(2,4),(3,4)\big)\in \N_4^\dd, \\
\MM&=(\P_{-1},\P_1,\P_{-1},\P_1).
}
Then $\MM_\alpha$ is of the form
\Eq{*}{
\MM_\alpha(x,y,z,t)&=\bigg(\frac{2xy}{x+y},\frac{x+y}2,\frac{2yt}{y+t},\frac{z+t}2\:\bigg).
}

\begin{figure}[h]
\centering
\captionsetup{justification=centering}
\begin{tikzpicture}[>={Stealth[width=6pt,length=9pt]},
node distance=15mm,auto]
\node[state,inner sep=3pt,minimum size=1pt] (b) at (0,0) {$2$};
\node[state,inner sep=3pt,minimum size=1pt] (c) at (2,0) {$3$};
\node[state,inner sep=3pt,minimum size=1pt] (d) at (2,2) {$4$};
\node[state,inner sep=3pt,minimum size=1pt] (a) at (0,2) {$1$};
\path[->] (a) edge [bend left] node {} (b);
\path[->] (a) edge [loop left] node {} (a);
\path[->] (b) edge [bend left] node {} (a);
\path[->] (b) edge [loop left] node {} (b);
\path[->] (b) edge node {} (c);
\path[->] (c) edge [bend left] node {} (d);
\path[->] (d) edge [bend left] node {} (c);
\path[->] (d) edge [loop right] node {} (d);
\end{tikzpicture}
\caption{Graph $G_\alpha$ related to Example~\ref{exa:Wcig2}.}
\end{figure}
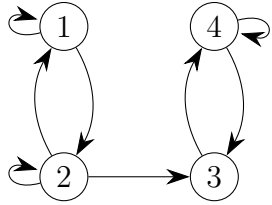

Analogously to Example~\ref{exa:dig}, we have
\Eq{*}{
\lim_{n \to \infty} \MM_\alpha^n(x,y,z,t)=(\sqrt{xy},\sqrt{xy},\sqrt{xy},\sqrt{xy}).
}
Therefore $K(x,y,z,t)=\sqrt{xy}$ is the only $\MM$-invariant mean (see \cite[Theorem~1]{MatPas21} for the detailed proof). Remarkably,
it depends on neither $z$ nor $t$.
\end{exa}

Finally, we show an example with periodic incidence graph, with the conjecture as in the case of weakly connected graphs.

\begin{exa}[Periodic incidence graph] \label{exa:pig}
Let $p=4$,
\Eq{*}{
\dd&=(2,2,2,2),\\
\alpha&=\big((3,4),(3,4),(1,2),(1,2)\big)\in \N_4^\dd, \\
\MM&=(\P_{-1},\P_1,\P_{-1},\P_1).
}
Then $\MM_\alpha$ is of the form
\Eq{*}{
\MM_\alpha(x,y,z,t)&=\bigg(\frac{2zt}{z+t},\frac{z+t}2,\frac{2xy}{x+y},\frac{x+y}2\:\bigg).
}
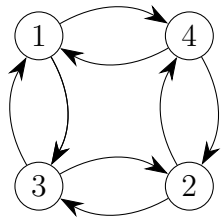
\begin{figure}[h]
\centering
\captionsetup{justification=centering}
\begin{tikzpicture}[>={Stealth[width=6pt,length=9pt]},
node distance=15mm,auto]
\node[state,inner sep=3pt,minimum size=1pt] (b) at (2,0) {$2$};
\node[state,inner sep=3pt,minimum size=1pt] (c) at (0,0) {$3$};
\node[state,inner sep=3pt,minimum size=1pt] (d) at (2,2) {$4$};
\node[state,inner sep=3pt,minimum size=1pt] (a) at (0,2) {$1$};
\path[->] (a) edge [bend left] node {} (c);
\path[->] (b) edge [bend left] node {} (d);
\path[->] (c) edge [bend left] node {} (a);
\path[->] (d) edge [bend left] node {} (b);
\path[->] (a) edge [bend left] node {} (d);
\path[->] (d) edge [bend left] node {} (a);
\path[->] (a) edge [bend left] node {} (c);
\path[->] (b) edge [bend left] node {} (c);
\path[->] (c) edge [bend left] node {} (b);
\end{tikzpicture}
\caption{Graph $G_\alpha$ related to Example~\ref{exa:pig}.}
\end{figure}

Then, after some computations, we have
\Eq{*}{
\MM_\alpha^2(x,y,z,t)&=\bigg(\frac{4xy(x+y)}{x^2+6xy+y^2},\frac{x^2+6xy+y^2}{4(x+y)},\\
&\qquad\qquad \frac{4zt(z+t)}{z^2+6zt+t^2},\frac{z^2+6zt+t^2}{4(z+t)}\:\bigg).
}
Analogously to Example~\ref{exa:dig} and the previous example, we have
\Eq{*}{
\lim_{n\to \infty} \MM_\alpha^{2n}(x,y,z,t)=(\sqrt{xy},\sqrt{xy},\sqrt{zt},\sqrt{zt}),
}
consequently
\Eq{*}{
\lim_{n\to \infty} \MM_\alpha^{2n+1}(x,y,z,t)&=\MM_\alpha \Big(\lim_{n\to \infty} \MM_\alpha^{2n}(x,y,z,t)\Big)\\
&= (\sqrt{zt},\sqrt{zt},\sqrt{xy},\sqrt{xy}).
}
Therefore $K(x,y,z,t)=S(\sqrt{xy},\sqrt{zt})$ is $\MM_\alpha$-invariant for every symmetric mean $S \colon \R_+^2 \to \R_+$. Moreover, using ideas from Example~\ref{exa:dig}, one can show that all continuous $\MM_\alpha$-invariant means are of this form.
\end{exa}

\subsection*{Funding}
The authors declare that no funds, grants, or other support were received during the preparation of this manuscript.
\subsection*{Competing Interests}
The authors have no relevant financial or non-financial interests to disclose.
\subsection*{Data availability}
Data sharing not applicable to this article as no datasets were generated or analysed during the current study.


\end{document}